\documentclass{article}
\usepackage{amsfonts}

\usepackage[utf8]{inputenc}
\usepackage[english]{babel}

\usepackage{tikz-cd}

\usepackage[all]{xy}
\usepackage[utf8]{inputenc}
\usepackage{color}

\usepackage{mathrsfs}
\usepackage[OT2, T1]{fontenc}
\usepackage{url}
\usepackage{amsmath}
\usepackage{mathabx}
\usepackage{array}
\usepackage{graphicx}
\usepackage{amssymb}
\usepackage{amstext}
\usepackage{amsthm}
\usepackage{hyperref}
\usepackage{colonequals}
\usepackage{enumitem}
\usepackage{cleveref}
\usepackage{fullpage}
\numberwithin{equation}{subsection}

\newtheorem{theorem}{Theorem}[section]
\newtheorem{lemma}[theorem]{Lemma}
\newtheorem{proposition}[theorem]{Proposition}
\newtheorem{corollary}[theorem]{Corollary}
\newtheorem{thm}{Theorem}

\theoremstyle{definition}
\newtheorem{definition}[theorem]{Definition}

\theoremstyle{remark}
\newtheorem{remark}[theorem]{Remark}

\DeclareMathOperator{\GL}{GL}

\DeclareMathOperator{\Spec}{Spec}

\newcommand{\ra}{\rightarrow}

\newcommand{\bC}{\mathbb{C}}

\newcommand{\F}{\mathbb{F}}

\newcommand{\Q}{\mathbb{Q}}
\newcommand{\ol}{\overline}

\newcommand{\R}{\mathbb{R}}

\newcommand{\V}{\mathbb{V}}
\newcommand{\Z}{\mathbb{Z}}

\newcommand{\cA}{\mathcal{A}}

\newcommand{\cD}{\mathcal{D}}

\newcommand{\cF}{\mathcal{F}}

\newcommand{\cL}{\mathcal{L}}

\newcommand{\cO}{\mathcal{O}}

\newcommand{\cS}{\mathcal{S}}
\newcommand{\cT}{\mathcal{T}}

\newcommand{\cW}{\mathcal{W}}
\newcommand{\cX}{\mathcal{X}}
\newcommand{\cY}{\mathcal{Y}}
\newcommand{\cZ}{\mathcal{Z}}

\newcommand{\fm}{\mathfrak{m}}

\newcommand{\et}{{\mathrm{et}}}

\newcommand{\an}{{\mathrm{an}}}
\newcommand{\ds}{{\displaystyle}}

\DeclareMathOperator{\Gal}{Gal}

\DeclareMathOperator{\Aut}{Aut}

\DeclareMathOperator{\ad}{ad}

\DeclareMathOperator{\rig}{rig}

\definecolor{ao(english)}{rgb}{0.0, 0.5, 0.0}

\title{CM points have everywhere good reduction}
\author{Benjamin Bakker, Jacob Tsimerman}

\begin{document}

\maketitle
\begin{abstract}
    We prove that for every Shimura variety $S$, there is an integral model $\cS$ such that all CM points of $S$ have good reduction with respect to $\cS$. In other words, every CM point is contained in $\cS(\ol\Z)$. This follows from a stronger local result wherein we characterize the points of $S$ with potentially-good reduction (with respect to some auxiliary prime $\ell$) as being those that extend to integral points of $\cS$.  
\end{abstract}

\section{Introduction}

The N\'eron-Ogg-Shafarevich criterion \cite{SerreTate} says that for an abelian variety over $\Q$, one can detect (potentially) good reduction at a prime $p$, by looking at whether the representation on its $\ell$-adic Tate module is (potentially) unramified, for $\ell\neq p$. One can reinterpert this result as being purely about the moduli space of (principally polarized) abelian varieties $\cA_g$ and the $\Z_\ell$-adic local system $\cL$ on it corresponding to the universal Tate module: \textit{A point $x\in\cA_g(\Q_p)$ is in the image of $\cA_g(\Z_p)$ iff the $\Gal_{\Q_p}$-representation $\cL_x$ is unramified.}

Our goal is to generalize this result to Shimura varieties which are not of abelian-type. One case of particular interest is that of CM points, and we record the following theorem:

\begin{thm}\label{Main:CM}
    Let $S=S_K(G,X)$ be a Shimura variety over $E=E(G,X)$. There is an integral model $\cS$ over $\cO_E$ for $S$ such that every CM point of $S(\ol\Q)$ extends to a point of $\cS(\ol\Z)$.
\end{thm}

Note that in \cite{PST} this was established away from finitely many `bad' primes, and the contribution of this note is to handle precisely that bad set.

This result is a consequence of the more precise result below. Let $V$ be a faithful $\Q$-representation of $G^{\ad}$, and $_{\et}V_\ell$ the corresponding $\ell$-adic local system. Let $v\nmid \ell$ be a finite place of $E$. For a local field $F\subset \ol{E_v}$ We say a point $x\in S(F)$ is \textit{$\ell$-potentially unramified } if the Galois representation $_{\et}V_{\ell,x}$ is potentially unramified. Let $S^{\ell-pun}_v\subset S(\ol E_v)$ denote the set of potentially unramified points. 

\begin{thm}\label{main}
     Let $S=S_K(G,X)$ be a Shimura variety over $E=E(G,X)$, and fix a  prime number $\ell$. There exists an integral model $\cS$ over $\cO_E$ such that for each finite place $v\nmid\ell$ of $E$, we have that
     $$S_v^{\ell-pun}=\cS(\cO_{\ol E_v}).$$
    Moreover, if $_{\et}V_{\ell}$ has torsion-free monodromy group, one may pick $\cS$ so that $_{\et}V_{\ell}$ extends to $\cS$.
\end{thm}

\begin{remark}
\vspace{0.1in}
    \begin{enumerate}
        \item If the monodromy group is not torsion-free then by our arguments one may still construct a Deligne-Mumford stack $\cS$ compactifying $S$ so that $_{\et}V_{\ell}$ extends. 
        \item It seems reasonable to conjecture that there is a model which works independently of $\ell$---indeed, this is true for `good' primes $v$, as we show below. It is the finitely many `bad' primes at which $S$ does not have a log-smooth model that we cannot show this.
        \item Indeed, even for proper Shimura varieties $S$, at `bad' finite places $v$ we cannot show that the $_{\et}V_{\ell,x}$ are potentially unramified for all $x\in S(E_v)$. Note that this would not follow simply from the existence of an appropriate motivic family over $S$, since such a family might a-priori have bad reduction at $v$.
        \item The proof of Theorem \ref{main} applies verbatim to any smooth variety $S/E$ with a $\Z_\ell$-local system $\cL$ which is non-extendable (or maximal) in the sense of \cite[Def 3.1]{brunebarbeshaf}, and such that $\cL_x$ is almost everywhere unramified for some (hence any) $x\in S(\ol{E})$. In particular, the latter condition holds for any local system of geometric origin.

    \end{enumerate}
\end{remark}
    
 \subsection{Acknowledgements}  Both authors benefited from many conversations with Ananth Shankar.  B. B. was partially supported by NSF grants DMS-2131688 and DMS-2401383.
\section{Background}

Let $\cO$ be a mixed characteristic $p$ DVR with fraction field $K$ and residue field $k$. 
 Let $\cX,\cX'/\cO$ be reduced separated schemes which are flat and finite type over $\cO$. 
 For an open subset $U\subset \cX$, a $U$-admissible modification $\cX'\to\cX$ (over $\cO$) is proper morphism which is an isomorphism over $U$.  If $U$ is the generic fiber, we simply refer to admissible modifications. 

 Let $X/K$ be a variety. A partial compactification for $X$ over $\cO$ is a reduced separated scheme $\cX$, flat and finite type over $\cO$ and a dense open embedding $X\subset \cX$ of $\cO$-schemes.  If $\cX_K\cong X$ we say $\cX$ is an integral model of $X$; if $\cX\to\Spec\cO$ is proper, we say $\cX$ is a compactification over $\cO$.

 If $R$ is a Dedekind domain with fraction field $K$, we say that a finite type $\cX/R$ is an integral model of $X$ if it is over every prime ideal of $R$.

\subsection{Chow's lemma}

We shall require the following version of Chow's lemma:

\begin{lemma}\label{lem:Chow}
Let $X/K$ be a quasiprojective variety and $\cX$ a compactification over $\cO$. 
 Then there exists a flat projective $f':\cX'\to\Spec\cO$ and an $X$-admissible modification $\psi :\cX'\to\cX$ over $\cO$.
\end{lemma}

\begin{proof}
    This is essentially \cite[\href{https://stacks.math.columbia.edu/tag/088R}{Tag 088R}]{stacks-project}. In the notation of that lemma, we let $Y=\Spec\cO, U=X$. Applying the lemma, we obtain
    a commutative diagram of $\Spec\cO$ schemes:

    \[
    \begin{tikzcd}
       & X\ar[dl]\ar[d]\ar[dr]\ar[drr]& & \\
    \cX& \cX'\ar[l,"f"]\ar[r,"\phi",swap]& \cZ'\ar[r,"g",swap] &\cZ\\ 
    \end{tikzcd}
    \]
    where the maps from $X$ are open immersions, $f$ is an $X$-admissible modification, $\phi$ is an open immersion, $g$ is an $X$-admissible modification and $\cZ$ is a projective $\Spec\cO$ scheme.

    Since blowups are projective, it follows that $\cZ'$ is projective over $\Spec\cO$, and since $\cX'$ is proper and $\phi$ is an open immersion it follows that $\phi$ is an isomorphism onto a closed subscheme of $\cZ'$, and hence $\cX'$ is also projective over $\Spec\cO$.

\end{proof}

\subsection{Integral models}

\begin{lemma}\label{etale descent model}
    Let $\cO'/\cO$ be an \'etale extension of DVRs, and $X/K$, $X'/K'$ normal quasiprojective varieties with an \'etale morphism $f:X'\to X$ over $\Spec K'\to \Spec K$.  Let $U\subset X$ be the subset over which $f$ is finite \'etale and $U':=f^{-1}(U)$.  Then for any compactification $\cX'_0$ of $X'$ over $\cO'$, there is an $U'$-admissible modification $\cX'\to\cX'_0$ and a projective compactification $\cX$ of $U$ over $\cO$ fitting into the commutative diagram below.
      \[\begin{tikzcd}
      U'\ar[d]\ar[rr]&&U\ar[d]\ar[rdd,bend left]&\\
      X'\ar[dd]\ar[rr,"f"{pos=.7}]&&X\ar[dd]&\\
      &\cX'\ar[rr,crossing over,"g"{pos=.3}]\ar[from=luu,bend left, crossing over]\ar[dd]&                &\cX\ar[dd]\\
      \Spec K'\ar[rd]\ar[rr]&&\Spec K\ar[rd]&\\
      &\Spec\cO'\ar[rr]&           &\Spec\cO
    \end{tikzcd}\]

    Moreover, $g$ is a quotient by a finite group action.

\end{lemma}
\begin{proof}
    Let $X''/K''$ and $\cX''_0/\cO''$ be the normalizations of $X'$ and $\cX_0'$ in the Galois closure of the function field extension of $X'\to X$, $G$ the Galois group, and $U''$ the preimage of $U'$.  By \Cref{lem:Chow}, up to replacing $\cX_0''$ with a $U''$-admissible modification we may assume $\cX''_0$ is quasiprojective.  Let $\cX_0^{\prime\prime G}$ be the $G$-fold fiber product of $\cX''_0$ over $\cO''$ equipped with the $G$-action by permuting factors.  Letting $\cX''\subset \cX_0^{\prime\prime G}$ be the closure of the natural $G$-equivariant embedding $U''\to \cX_0^{\prime\prime G}$, any projection $\cX''\to \cX_0''$ is a $U''$-admissible modification.  Letting $\cX$ (resp. $\cX'$) be the quotient of $\cX''$ by $G$ (resp. the Galois group of $X''/X'$), we obtain the required $\cX$ (resp. $\cX'$).  
\end{proof}

\begin{corollary}\label{complete descent model}
    Let $\hat \cO$ be the completion of $\cO$ with fraction field $\hat K$, $X/K$ a quasiprojective variety, and $\hat \cX$ an integral model of $X_{\hat K}$.  Then there exists a quasiprojective model $\cX$ of $X$ and an admissible morphism $\cX_{\hat \cO}\to \hat \cX$.
\end{corollary}

\begin{proof}
By \cite[Claim 3.1.3.1]{vasiu} the model $\hat \cX$ descends canonically to a model $\cX'_0$ of $X_{K'}$ defined over an \'etale DVR extension $\cO'/\cO$.  Now apply Lemma \ref{etale descent model}.
\end{proof}
    
Now assume that $K$ is a local field.  Let $D\subset X$ be a closed subscheme. We let $X^{\an}$ denote the Berkovich space \cite{ber} associated to $X$ through analytification and let $\cX$ denote an integral model of $X$.

The following is almost certainly well known, but we could not find a reference so provide our own proof:

\begin{lemma}\label{lem:justtube}
   Assume $\cX/\cO$ is proper. For any neighborhood $U$ of $D^{\an}$, there exists an admissible modification $\cY\to \cX$ such that if $\cD$ is the closure of $D$ in $\cY$, then the complement $U^c$ is contained in the rigid fiber of $\cY\backslash \cD$. 
\end{lemma}
\def\rig{\mathrm{rig}}
\begin{proof}

Setting $\cX=\cX_0$, we form a tower of models $\pi_{n+1}:\cX_{n+1}\to \cX$ inductively, by letting $D_n$ be the closure of $D$ in $\cX_n$ and blowing up the special fiber $(D_{n})_k$. We claim that for large enough $n$, the model $\cX_n$ suffices. To prove this, it is sufficient to work locally on $\cX$, so we consider an affine open subset $\Spec R\subset \cX$ where $D\cap\Spec R$ has defining ideal $(f_1,\dots,f_n)$.

Inductively then, $\cD_m\cap \pi_m^{-1}\Spec R$ lies in a single affine chart, and is cut out by $(\frac{f_1}{\pi^m_K},\cdots,\frac{f_n}{\pi^m_K})$, and so a valuation  $w\in (\Spec R)^{\rig}$ maps to $\cD_m$ iff $$\max_i|f_i(Q)|_w< |\pi_K^m|.$$ For a sufficiently large $n$ this lies within our tube $U$ by continuity, completing the proof. 
    
\end{proof}

\subsection{Local systems}

We refer to \cite{milnemoduli,milneintro} for background on Shimura varieties.

Let $(G,X)$ be a Shimura datum satisfying the axioms in \cite{milnemoduli}, and $S=S_K(G,X)$ be a Shimura variety with reflex field $E=E(G,X)$.  Let $V$ be a faithful $\Q$-representation of $G^{\ad}$, and $\V\subset V$ be a $K$-stable lattice. For a prime number $\ell$, we  let $_{\et}\V_\ell$ denote the corresponding $\Z_\ell$-local system. Furthermore, there exists a natural $\Z$-local system $_B\V$ on $S(\bC)$ underlying a variation of Hodge structures, with proper corresponding period map. In other words, if $\ol{S}$ is a log-smooth compactification of $S=S_K(G,X)$, then $V$ has infinite monodromy around every irreducible component of $D:=\ol{S}\backslash S$ (see \cite[Theorem 9.5]{griffiths1970periods}.

We shall require the following lemma:

\begin{lemma}\label{lem:hodgefaithful}
    Let $\cL$ be a local system underlying a VHS on a smooth complex variety $X$ with log-smooth compactification $\ol{X}\backslash X$, and assume the monodromy around each boundary divisor is unipotent. Let $Q\in \ol{X}$ and let $F_1,\dots, F_m$ be the irreducible divisors of  $D:=\ol{X}\backslash X$ containing $Q$. Let $q\in X$ be an infinitesimally nearby point to $Q$ and consider the monodromy elements $\sigma_1,\dots,\sigma_m\in \pi_1(X,q)$ corresponding to the simple loops around $F_1,\dots,F_m$. Then there is no nontrivial  product $\ds\prod_i\sigma_i^{\Z_{\geq 0}}$ which acts trivially on $\cL_q$.
\end{lemma}

\begin{proof}
    This follows immediately from \cite[Proposition 9.11, Theorem 9.5]{griffiths1970periods}.
\end{proof}



\subsection{Local \'etale fundamental groups of semistable schemes}

\begin{lemma}\label{puncture}
Let $(R,\fm)$ be a complete regular local ring of mixed characteristic $p$, with algebraically closed residue field, and with a regular system of parameters $(x_1,\dots,x_n,y_1,\dots,y_m)$. Let 
$S=R\left[\frac1y_1,\dots\frac1y_m\right]$. Then the maximal prime-to-$p$ Galois \'etale extension $S_0$ of $S$ is generated by the prime-to-$p$ roots of the $y_i$.

\end{lemma}

\begin{proof}

 Let $W$ be a Galois \'etale extension of $S$ of degree prime-to-$p$. Let $T$ denote the normal closure of $R$ in $W$. Now for each $j\in\{1,\hdots,m\}$ we let $R_j$ and $ T_j$ denote the localization of $R$ and $ T$ at the prime ideals $(y_j)$ and $\mathfrak{y}_j$ respectively, where $\mathfrak{y}_j$ is some prime ideal of $T_j$ sitting above $(y_j)$. Now $T_j,R_j$ are discrete valuation rings. Let $e_j$ denote the ramification degree and let $e=\displaystyle\prod_j e_j$. Now let $W'$ denote the compositum of $W$ and the $e$'th roots of all the $y_i$, and let $T', T'_j$ be as before. Then by Abhyankar's lemma \cite[\href{https://stacks.math.columbia.edu/tag/0BRM}{Tag 0BRM}]{stacks-project}, $T'_j$ is unramified over $R[y_1^{\frac1{e}},\dots,y_m^{\frac1{e}}]_j$ for all $j$. By the purity of the branch locus \cite[\href{https://stacks.math.columbia.edu/tag/0BMB}{Tag 0BMB}]{stacks-project} it follows that $T'$ is unramified over $R[y_1^{\frac1e},\dots,y_m^{\frac1e}]$. This latter ring is complete with algebraically closed residue field, and so $T'=R[y_1^{\frac1e},\dots,y_m^{\frac1e}]$. The claim is thus proven.
\end{proof}

\subsection{Descending local systems under finite maps}

\begin{proposition}\label{prop:groupextend}
    Let $K$ be a local field with residue field $k$, and assume $k$ is perfect. Let $\cY/\cO_K$ be finite type, separated, flat, quasi-projective, and assume $\cY$ is normal. Let $G$ be a finite group acting on $\cY$, and let $f:\cY\ra \cX$ be the quotient. Assume the action of $G$ is free on $Y=\cY_K$.  

    Let $\cL$ be a $\Z_\ell$-local system on $X$ with torsion free monodromy. If $\cL_Y$ extends to $\cL_\cY$, then $\cL$ extends to $\cX$. 
\end{proposition}

\begin{proof}

Let $y\in \cY$ be a closed point, and $I<G$ the stabilizer of $i_y:\Spec k(y)\ra \cY$.  

Let $X_{\cL}$ be the pro-scheme corresponding to the monodromy of $\cL$, with $M:=\Aut(\cY_\cL/\cY)$. Let $Y_\cL:=Y\times_X X_{\cL}$ and $\cY_{\cL}$ the normalization of $\cY$. Note by normality of $\cY$ that $\phi:\cY_{\cL}\ra\cY$ is Galois. Therefore $M\times G$ acts on $\cY_{\cL}$, and $M\times I$ acts on $\phi^{-1}(\ol y)$, while $M$ acts simply transitively on this set. Since there are no non-trivial group homomorphisms from $I$ to $M$ it follows that $I$ acts trivially on $\phi^{-1}(\ol y)$.

Let $\cY_n$ be a subcover corresponing to a finite quotient $M_n$ of $M$. Now let $z\in \cY_n$ such that $\phi(z)=y$. By the above, $I$ acts trivially on $k(z)$. We claim that $\cO_z^I$ is \'etale over $\cO_y^I$. To check this is is enough to pass to the completion. But then $\widehat{\cO_z}\cong \widehat{\cO_y}\otimes_{W(k(y))} W(k(z))$, and $I$ acts trivially on $W(k(z))$. It follows that
$\widehat{\cO_z}^I=\widehat{\cO_y}^I\otimes_{W(k(y))} W(k(z))$ and the claim follows.

Finally, it follows that $\cY_n/G$ is finite etale over $\cY/G\cong \cX$. Thus the inverse limit of the $\cY_n/G$ give an extension of $\cL$ to $\cX$ as desired.

\end{proof}


\section{Proof of Theorem \ref{main}}

For some positive integer $N$ divisible by $\ell$, we have that $S$ spreads out to $\cS$ with compactification $\cS\subset \cS'$ over $\cO_E[\frac1N]$. We let $T$ be the cover of $S$ corresponding to the level where the representation defined by  $_{\et}\V_\ell$ is trivial mod $\ell^2$, and we assume $T$ spreads out to a smooth model $\cT$ over $\cO_E[\frac1N]$ which is a Galois cover of $\cS$, and let $\pi:\cT\ra\cS$ be the natural map. Let $Q_{CM}\in T(\ol \Q)$ be a CM point. By increasing $N$, we assume further that $\cT\subset \cT'$ is a log-smooth compactification, that the Zariski-closure of $Q_{CM}$ in $\cT'$ is contained in $\cT$, and that the field of definition of $Q_{CM}$ is unramified away from $N$. Note that by our assumption, the monodromy elements around all the boundary divisors are unipotent.

The following lemma is essentially the same argument as \cite[Lemma 4.4]{PST}:

\begin{lemma}\label{lem:almostall}
    For all finite places $v\nmid N$ of $E$,  we have $S_v^{\ell-pun}=\cS( \ol{\cO_{E_v}})$.
\end{lemma}

\begin{proof}

    Let $\cL:=\pi^*_{\et}V_\ell$ be the corresponding $\ell$-adic local system on $\cT$. We claim that $\cL$ must have trivial monodromy around the special fiber of $\cT$, which follows by Lemma \ref{puncture} since $Q_{CM}$ has everywhere potentially good reduction. Thus, $\cL$ must extend to $\cT$. We let $K=E_v$.
    
    Note that $\pi^*\cS( \ol{\cO_{E_v}})=\cT( \ol{\cO_{E_v}})$ and so it suffices to prove that 
    $T_v^{\ell-pun}= \cT( \ol{\cO_{E_v}})$ where $T_v^{\ell-pun}$ is defined relative to $\cL$. So suppose that $Q\in T_v^{\ell-pun}(\ol{E_v})\backslash \cT( \ol{\cO_{E_v}})$ and let $z\in \cT'(\ol\F_v)$ be the reduction of $Q$ at the special fiber.  Finally, we denote by $\cT'_{un}$ the base change of $\cT'$ to the maximal unramified extension $K_{un}$ of $K$.

    Let $x_1,\dots,x_n\in R:=\widehat{\cO}_{\cT',z}$ be a regular sequence cutting out the irreducible components of the boundary divisors of $\cT'\backslash\cT$ at $Q$, and let $\sigma_1,\dots,\sigma_n$ be the generators of the natural $\Z_\ell(1)^n$ quotient of $\pi_1(R[\frac{1}{x_1},\ldots,\frac{1}{x_n}])$ corresponding to the complex loops around those same boundary divisors with respect to an identification $\ol{K}\cong\bC$.
    
    Identifying the maximal $\ell$-power quotient of $\pi_1(\cO_{K_{un}})$ with $\Z_\ell(1)$, the natural map
    $$\pi_1\left(R\left[\frac{1}{x_1},\ldots,\frac{1}{x_n}\right]\right)\ra \pi_1(\cO_{K_{un}})$$ is naturally identified with 
    $$(a_1,\dots,a_n)\ra \displaystyle\sum_{i=1}^n a_iv_Q(x_i).$$ Thus, since $Q\in T_v^{\ell-pun}$, we conclude that $\prod_{i=1}^n \sigma_i^{a_i}$ acts trivially on $\cL_Q$. However, this contradicts Lemma \ref{lem:hodgefaithful}, which completes the proof.

\end{proof}
Since we may simply glue in integral models at finitely many places, by \Cref{complete descent model} we have reduced ourselves to proving the following:

\begin{proposition}\label{prop:badprimes}
    Let $v\nmid \ell$ be a finite place of $E$. There exists an integral model $\cS$ of $S_v$ over $\cO_{E_v}$ such that $S_v^{\ell-pun}= \cS( \ol{\cO_{E_v}})$. Morover, if $\cL$ has torsion-free monodromy, we may choose $\cS$ so that $\cL$ extends.
\end{proposition}

We first reduce to the case where the monodromy of $_{\et} \V_\ell$ is trivial mod $\ell$ (and hence pro-$\ell$, torsion free and prime-to-$p$). To that end, let $f:W\ra S$ denote the $\ell^2$-level cover of $S$.

\begin{proposition}\label{lem:redtouni}
    Proposition \ref{prop:badprimes} for $W$ implies Proposition \ref{prop:badprimes} for $S$.
\end{proposition}

\begin{proof}
Assume that Proposition  \ref{prop:badprimes} is true for $W$ with the model $\cW$. By Lemma \ref{lem:Chow} we may assume $\cW$ is projective. Let $G$ be the Galois group of $f$, and consider the map $h:W\ra \cW^G$ given by $h(w)_g=g\circ f(w)$. Finally, let $\cW_1$ denote the normalization of the closure of $h(W)$. Then $\cW_1$ is also an integral model of $W$, and from construction $W_v^{\ell-pun}=\cW_1(\ol{\cO_{E_v}})$. Moreover there is a natural group action of $G$ on $\cW_1$, and the quotient $\cS_1$ is an integral model for $S$.

Finally, we claim that $\cL$ extends to $\cS$. Since $\cS$ is normal it is enough to prove this locally. But the image of inertia around every point must be simultaneously torsion and yet contained in the monodromy image of $\cL$, which is torsion-free. Hence the image of inertia is trivial, which means that $\cL$ extends.
\end{proof}
Henceforth we take $K=E_v$, $k=\F_v$, and suppress the subscript $v$ in $S_v$.  By \Cref{etale descent model} we may pass to the $\ell^2$-level cover, and therefore assume $\cL:=_{\et} \V_\ell$ has unipotent local monodromy. Let $\ol{S}$ denote a log-smooth compactification of $S$ over $K$. To streamline the proof we introduce the following terminology:  

\begin{definition}\label{discerning}
Let $\phi:T\ra \ol S$ be a proper $\ol S$-scheme, and let $\cT$ be a model of $T$ over $\cO_{K}$. We say that $\cT$ is \textit{discerning} if there is a \textit{distinguished} subset $Z\subset \cT(\ol k)$ such that
\begin{enumerate}
    \item for all $x\in \cT(\ol{\cO_{K}})$ such that $\phi(x_{\ol{K}})\in S(\ol{K})$, the local system $\phi_{x_{\ol K}}^* \cL$ is potentially unramified iff $x_{\ol k}\not\in Z$. 
    \item $Z$ contains the $\ol k$-points of the closure of $\phi^{-1}(\ol S\backslash S)$.
\end{enumerate}
\end{definition}    

Note that if $\cT$ is discerning then any $\cT$-scheme is discerning as well.  Also, since every $\ol k$-point of $\cT$ lifts to a $\ol{\cO_K}$-point of $\cT$ whose $\ol K$-point is contained in $S$, $Z$ is uniquely determined.




\begin{proposition}\label{prop:trueforreg}
    Let $\phi:X\to \ol S$ be a proper $\ol{S}$-scheme which is smooth over $K$, with a semistable compactification $\cX'$ over $\cO_K$. 
    
    \begin{enumerate} 
    
        \item If $\cX'$ itself is discerning, then the distinguished set $Z\subset \cX'(\ol k)$ is closed.
    
        \item There exists an admissible modification of $\cX'$ which is discerning. Moreover, if the monodromy-image of $\cL$ is torsion-free, we may pick this modification such that $\cL$ extends to the complement of $\phi^{-1}(\ol S\backslash S)\cup \ol Z$ in $\cX'$.
       
    \end{enumerate}
\end{proposition}

\begin{proof}

    Let $W\subset \cX'$ be a (locally closed) boundary stratum contained in the special fiber. Let $F_1,\dots,F_n$ be the irreducible boundary divisors not contained in the special fiber which contain $W$, $G_1,\dots,G_m$ the irreducible boundary divisors in the special fiber containing $W$, and $H$ the union of the boundary divisors not containing $W$. Let $z\in W$ be the generic point. Finally, let $x_1,\dots,x_m,y_1,\dots,y_n$ be a subset of the regular sequence of generators of $R=\widehat{\cO}_{\cX',z}$ cutting out $F_1,\dots,F_m,G_1,\dots,G_n$. 

Let  $\sigma_1,\dots,\sigma_m,\tau_1,\dots,\tau_n$ be the generators of the natural $\Z_\ell(1)^{m+n}$ quotient of $\pi_1(R[\frac{1}{x_1},\ldots,\frac{1}{x_m},\frac{1}{y_1},\ldots,\frac{1}{y_n}])$ corresponding to the complex loops around the boundary divisors.  Let $Q\in X(\ol K)$ be a point mapping to $S$ and extending to a point in $\cX'(\ol{\cO_K})$ whose reduction lands in $W$. By Lemma \ref{puncture},  we have that
$\prod_{i=1}^m \sigma_i^{v_Q(x_i)}\cdot \prod_{j=1}^n \tau_j^{v_Q(y_j)}$ acts trivially on $_{\et} \V_{\ell,z}$ iff $Q\in S^{\ell-pun}$.

Let $\phi:\pi_1(R[\frac{1}{x_1},\ldots,\frac{1}{x_m},\frac{1}{y_1},\ldots,\frac{1}{y_n}])\ra \GL(\cL_z)$ denote the monodromy map. Since the image of $\phi$  is $1\mod \ell$ and abelian, we may take logarithms and conclude that $Q\in S^{\ell-pun}$ iff 
$$\sum_{i=1}^m v_Q(x_i)\log\phi(\sigma_i) +  \sum_{j=1}^n v_Q(y_j)\log\phi(\tau_j)=0.$$ 

For part 1, note that for $z\in W(\ol k)$ we may pick a lift $Q$ with any positive integral values of the $v_Q(x_i),v_Q(y_j)$. Therefore to be a discerning model, it must be true that either all the $\log\phi(\sigma_i)$ are $0$, or that there are no linear relations between them with $\Q_{>0}$-coefficients. Moreover, this must be true for every stratum $W$. If this is satisfied, $Z$ is simply the union of the $G_i(\ol k)$ for which the corresponding element $\log \phi(\sigma_i)$ vanishes, and is therefore closed.

We now prove part 2. What follows is a version of \cite[Proposition 3.5]{brunebarbeshaf}.  Consider the set $P$ of non-negative rational solutions $(a,b)\in\Q_{\geq 0}^{m+n}$ to 
$$\sum_{i=1}^m a_i\log\phi(\sigma_i) +  \sum_{j=1}^n b_j\log\phi(\tau_j)=0$$
According to \Cref{lem:hodgefaithful}, the intersection of $P$ with the sub-cone $\Q_{\geq 0}^m\times 0$ corresponding to $F_1,\ldots, F_m$ is 
$0$.  We may therefore find an integral subdivision $\cF$ of the standard fan on $\R_{\geq0}^{m+n}$ for which $P$ is a union of cones and which does not refine the standard fan on the sub-cone $\R_{\geq 0}^m\times 0$.
If $\ol W$ is the closure of $W$, then this yields a $\cX'\setminus \ol W$ admissible monomial modification $\phi_W:\cX''_W\ra \cX'$. For the point $Q\in X(\ol K)$ specializing to $W$ as before, we see that $Q\in S^{\ell-pun}$ iff its specialization in the special fiber of $\cX''_W$ is contained in at least one divisor corresponding to a one-dimensional cone of $\cF$ not contained in $P$.  Let $Z_W\subset \cX''_W(\ol k)$ be the set of $\ol k$-points of the special fiber of the union of these divisors, and note that the special fiber of the strict transform of each $F_j$ is contained in $Z_W$.  

Let $\cX''\to\cX'$ be an admissible modification with $\cX''$ normal which factors as $\cX''\xrightarrow{\pi_W}\cX''_W\to\cX'$ for each $W$.  Then it follows that $\cX''$ is discerning with distinguished set $Z:=\bigcup_W \pi_W^{-1}(Z_W)$. 

It remains to show that $\cL$ extends if the monodromy is torsion-free.  Since $\cX''$ is normal, it suffices to show this locally.  Note that each fan $\cF$ may be chosen to be simplicial, so that $\cX''_W$ is semistable, in which case it is clear from the construction that $\cL$ extends to $\phi_W^{-1}(W)$. 



\end{proof}

\begin{corollary}\label{cor:distinguishedclosed}

    If $\cT$ is discerning, the distinguished set $Z\subset \cT(\ol k)$ is closed.
\end{corollary}

\begin{proof}
    By \cite[Thm 4.5]{BS}, there is a proper morphism $f:X\ra \ol{S}$ from a smooth proper scheme $X/K$ which admits a semistable compactification $\cX'$. By \ref{prop:trueforreg}, the set $f^{-1}(Z)$ is closed. Since $f$ is proper it follows that $Z$ is closed.
\end{proof}

Given a discerning model $\cT$ of $\phi:T\ra \ol S$, we define $\cT^\circ$ to be the complement of $\phi^{-1}(\ol S\backslash S)\cup\ol Z$.

\begin{proposition}\label{prop:Maincover}
    There exist finitely many discerning partial compactifications $\cS_1,\dots,\cS_m$ of $\ol{S}$ such that  every $\ol{K}$-point of $\ol{S}$ extends to an $\ol{\cO_{ K}}$-point of $\cS_i$ for at least some $i$. Moreover, if the monodromy-image of $\cL$ is torsion-free, we may pick the $\cS_i$ such that $\cL$ extends to $\cS_i^\circ$.
   
\end{proposition}

\begin{proof}\hspace{.5in}
\vskip1em\noindent
\emph{Step 1.}  
By \cite[Thm 4.5]{BS}, for any finite set of points $F\subset \ol S(\ol{K})$ there is a proper morphism $f:X\ra \ol{S}$ from a smooth proper scheme $X/K$ which is finite \'etale over a neighborhood of $F$ such that $X$ admits a semistable compactification $\cX_i'$ over $\ol{\cO_{K}}$.  It follows by Noetherian induction that there is a finite set $\{X_i\}_{ i\in I}$ of such $\ol S$-schemes such that $f_i:X_i\to \ol S$ is finite \'etale over an open $V_i\subset \ol S$ and $\{V_i\}_{i\in I}$ is an open cover of $\ol S$.

\vskip1em\noindent
    \emph{Step 2.} For each $i\in I$, let $\cX_{i,0}'$ denote a discerning admissible modification of $\cX_i'$, using Proposition \ref{prop:trueforreg}. By \Cref{etale descent model}, for each $i$ there is a $f_i^{-1}(V_i)$-admissible modification $\cX_i''\to\cX_{i,0}'$, a projective compactification $\cY_i$ of $V_i$, and a quotient $\cX_i''\to \cY_i$ by a finite group $G_i$ which is \'etale on $f_i^{-1}(V_i)$.  The $\cX_i''$ are also discerning.  After taking normalizations we may assume $\cX_i''$ and $\cY_i$ are normal.  Since the morphism $X_i''\to \ol S$ factors through $X_i''\to Y_i$ on the level of function fields, it follows it factors through a $V_i$-admissible modification $\pi_i:Y_i\to\ol S$.  Thus, $\cY_i$ is discerning.  Let $D_i\subset \ol S$ be the open subset over which $\pi_i$ is not an isomorphism; since $ V_i\subset \ol S\setminus D_i$, we have $\bigcap_{i\in I}D_i=\varnothing$.
\vskip1em\noindent
    \emph{Step 3.} Let $U_i\subset \bar S^\an$ be neighborhoods of the $D_i^{\an}$ which have no intersection. By Lemma \ref{lem:justtube}, there is an admissble modification $\cY_i'\to \cY_i$ such that, setting $\pi_i':Y_i'\to \ol S$ to be the natural map, the complement of $\pi_i'^{-1}(U_i)$ is contained in the rigid fiber of $\cY_i'\backslash \ol{\pi_i'^{-1}(D_i)}$.  Let $\cX_i^\dagger\to \cY_i'$ be the normalization of the reduction of the base-change of $\cX''_i\to \cY_i$. Then $\cY_i'$ is the quotient of $\cX_i^\dagger $ by $G_i$ since both are normal.  We set $\cS_i:=\ol S\cup_{\ol S\backslash D_i}\cY_i'\backslash \ol{\pi_i'^{-1}(D_i)}$.  By construction, every $\cS_i$ is discerning, since every integral point of $\cS_i$ is an integral point of $\cY_i'\backslash \ol{\pi_i'^{-1}(D_i)} $.  Moreover, every $\ol {K}$-point of $\ol S$ is contained in the complement of one of the $U_i$, hence extends to an $\ol{\cO_{K}}$ point of $\cS_i$.
    \vskip1em\noindent\emph{Step 4.} Finally, if the monodromy is torsion-free, then by Proposition \ref{prop:trueforreg} we may pick $\cX'_{i,0}$ such that $\cL$ extends to $\big(\cX'_{i,0}\big)^\circ$ and therefore $\big(\cX^{\prime\prime}_i\big)^{\circ}$ and $(\cX_i^{\dagger})^\circ$, in which case it'll extend to each of the $(\cY_i')^{\circ}\backslash \ol{\pi_i'^{-1}(D_i)} $ by Proposition \ref{prop:groupextend}, and finally to $\cS^{\circ}_i$ by construction. This completes the proof.

\end{proof}

We now complete the proof of Proposition \ref{prop:badprimes}. Let $Z_i\subset \cS_{i}$ denote the closed subschemes as in \Cref{discerning}. By Nagata's compactification theorem we may embed each $\cS_i$ in a proper $\cS'_i$, which is therefore a proper model of $\ol{S}$.

Let $\cS_0$ be a normal model of $\ol S$ with admissible modifications $g_i:\cS_0\ra \cS'_i$ for each $i$. We claim that $\cS_0$ is a discerning. Define $Z_0\subset \cS_0$ by $Z_0:=\bigcup_{i\in I}g_i^{-1}(Z_i)$.  Clearly $Z_0$ contains $\ol S\setminus S$.  A $\bar K$-point $x$ of $S$ extends to an integral point of $\cS_0$, and the extension specializes to $Z_0$ if and only if $x$ extends and in some $\cS_i$ and specializes to $Z_i$, which is the case if and only if it is not potentially unramified.   

By Definition \ref{discerning} and Corollary \ref{cor:distinguishedclosed}, the set $\ol{Z_0}\bigcup(\ol S\backslash S)$ is closed.  By taking $\cS$ to be the complement of $\ol{Z_0}\bigcup(\ol S\backslash S)$ in $\cS_0$, we obtain the desired model of $S$.  

Finally, observe by construction and Proposition \ref{prop:Maincover} that if the monodromy is torsion-free, then $\cL$ extends locally across every point of $\cS$. Since $\cS$ is normal this extension is canonical, and hence $\cL$ extends to all of $\cS$.  This completes the proof of Proposition \ref{prop:badprimes}, and thus Theorem \ref{main}.\qed

\section{Proof of Theorem \ref{Main:CM}}

 By \cite[\S12]{milnemoduli} all CM points lie in $S^{\ell-pun}_v$ for all $v\nmid\ell$. Let $\ell_1$ be an odd rational prime inert in $E$. Let $\cS_1,\cS_2$ be the integral models in Theorem \ref{main} corresponding to $\ell=\ell_1,\ell_2$ respectively. We may simply form $\cS$ by gluing $\cS_1\times_{\cO_E}\cO_E[\frac{1}{\ell_1}]$ with $\cS_2\times_{\cO_E}{\cO_{E,(\ell_1)}}$ along $S$, and the theorem is proved. \qed

  \bibliography{main}
  \bibliographystyle{plain}




\end{document}